\documentclass[onecolumn,draftcls,12pt, twoside]{IEEEtran}

\usepackage{main}

\newcommand{\tf}{d}

\newcommand{\nuser}{n}
\newcommand{\ecri}{L}

\newcommand{\E}{\mathrm{E}}
\renewcommand{\P}{\mathrm{P}}
\newcommand{\gk}[1]{\left\{#1\right\}}
\newcommand{\hk}[1]{^{(#1)}}

\newcommand{\ek}[1]{\left[#1\right]}
\newcommand{\rk}[1]{\left(#1\right)}
\newcommand{\Pfrak}{\mathfrak{P}}
\renewcommand{\l}{\lambda}
\newcommand{\ex}{\mathrm{e}}
\newcommand{\Ocal}   {{\mathcal O }} 
\renewcommand{\d}{\mathrm{d}}
\newcommand{\mgfl}{Q} 
\newcommand{\tmgfl}{\widetilde{Q}}
\newcommand{\scip}{\mathrm{M}}
\newcommand{\1}{\mathbbm{1}\!}
\newcommand{\C}{\mathbb{C}}
\newcommand{\N}{\mathbb{N}}
\newcommand{\R}{\mathbb{R}}
\newcommand{\Fp}{\mathrm{F}}
\newcommand{\Fq}{\overline{\mathrm{F}}}
\newcommand{\Mcal}   {{\mathcal M }} 
\newcommand{\ps}{p^\mathrm{bi}}
\newcommand{\wpi}{\widetilde{\pi}}
\newcommand{\poi}{\mathrm{Poi}}
\newcommand{\Ucal}{\mathrm{Unif}}

\begin{acronym}
\acro{AP}{access point}
\acro{RA}{random access}
\acro{LT}{Luby Transform}
\acro{BP}{belief propagation}
\acro{i.i.d.}{independent and identically distributed}
\acro{IoT}{Internet of things}
\acro{MTC}{machine type communication}
\acro{PER}{packet error rate}
\acro{SIC}{successive interference cancellation}
\acro{URLLC}{ultra-reliable low latency communication}
\acro{PMF}{probability mass function}
\acro{CRI}{collision resolution interval}
\acro{PGF}{probability generating function}
\acro{CPGF}{conditional probability generating function}
\acro{EGF}{exponential generating function}
\acro{TGF}{transformed generating function}
\acro{rhs}{right-hand side}
\acro{lhs}{left-hand side}
\acro{CSA}{coded slotted ALOHA}
\acro{MPR}{multi-packet reception}
\acro{MTA}{modified binary-tree algorithm}
\acro{SICTA}{binary tree-algorithm with successive interference cancellation}
\acro{BTA}{binary tree-algorithm}
\acro{MST}{maximum stable throughput}
\acro{MTA}{modified tree-algorithm}
\acro{CRP}{collision-resolution protocol}
\acro{CAP}{channel-access protocol}
\acro{IRSA}{Irregular Repetition Slotted ALOHA}
\acro{CSA}{Coded Slotted ALOHA}
\acro{STA}{standard tree-algorithm}

\acro{r.v.}{random variable}
\end{acronym}

\title{Analysis of $d$-ary Tree Algorithms with Successive Interference Cancellation}

\author{\IEEEauthorblockN{Quirin Vogel\IEEEauthorrefmark{1}, Yash Deshpande\IEEEauthorrefmark{2}, \v Cedomir Stefanovi\' c\IEEEauthorrefmark{3}, Wolfgang Kellerer\IEEEauthorrefmark{2}}\\
\IEEEauthorblockA{
\IEEEauthorrefmark{1}Department of Mathematics, School of Computation, Information and Technology, Technical University of Munich, Germany\\
\IEEEauthorrefmark{2}Chair of Communication Networks, School of Computation, Information and Technology, Technical University of Munich, Germany \\
\IEEEauthorrefmark{3}Department of Electronic Systems, Aalborg University, Denmark \\
Email: \{quirin.vogel, yash.deshpande,wolfgang.kellerer\}@tum.de }, cs@es.aau.dk}

\begin{document}

\maketitle

\begin{abstract}
In this article, we calculate the mean throughput, number of collisions, successes, and idle slots for random tree algorithms with successive interference cancellation. Except for the case of the throughput for the binary tree, all the results are new. We furthermore disprove the claim that only the binary tree maximises throughput. Our method works with many observables and can be used as a blueprint for further analysis.
\end{abstract}

\begin{IEEEkeywords}
medium access algorithms, wireless communications, random access, 5G, Asymptotic Analysis, Functional Equation, Splitting algorithms.
\end{IEEEkeywords}

\acresetall
\section{Introduction}
\label{sec:introduction}
Tree algorithms address the classical \ac{RA} problem where $n$ users each want to transmit a packet to a common receiver. 
Users can only communicate with the receiver and not among themselves. 
Tree algorithms solve the problem by iteratively splitting users into different groups until each group has only one user. 
Each group then transmits in a time slot determined by the algorithm. 
A metric of a tree algorithm's efficiency is the ratio between $n$ and the number of time slots time it takes until all users have successfully transmitted their packet, called the throughput.
Yu and Giannakis introduced the \ac{SICTA} in~\cite{yu2007high}. 
SICTA extends previous tree algorithms and offers high throughput.
The key idea of SICTA is to, once they become decoded, successively remove user packets along the tree; this way, some of the previous groups may become reduced to having just a single user, propelling a new round of decoding and \ac{SIC}. 
In this work, we analyse the throughput of SICTA, as well as the mean number of collisions, successes, and idle slots, for the general version of the algorithm in which the users randomly split into $d$ groups, $d \geq 2$. 

The rest of the paper is organized as follows. In Section \ref{SubsectionMathIntro}, we give a brief overview of the novelties contained in our work. In Section \ref{sec:results}, we first give a brief mathematical description of the model, for readers unfamiliar with SICTA. We then state the main results. In Section \ref{sec:backgound} we provide background on tree-splitting algorithms and mention the related work (see Section \ref{subsectionliterature}). In Section \ref{sec:analysis}, we prove our results, i.e., we derive the correct expression for the steady state \ac{CRI} length conditioned on the number $\nuser$ of initially collided users for $\tf$-ary SICTA algorithm. We then give the asymptotic expressions for throughput, number of collision slots and number of immediately decodable slots (henceforth referred to successes) when the number of users $\nuser$ tends to infinity. We also derive results on the mean delay experienced by a user. 

\subsection{Overview of the our contributions}\label{SubsectionMathIntro}

Compared to other tree algorithms such as \ac{STA} and \ac{MTA}, studying the properties of SICTA requires a more careful approach. Indeed, \ac{SIC} (see Section \ref{sec:backgound}) introduces further dependencies into the model, which are non-trivial, especially in the case $\tf\ge 3$. 
These subtle dependencies caused errors in the literature \cite{yu2007high}, which were identified in \cite{deshpande2022correction}. 
However, \cite{deshpande2022correction} does not provide the formal analysis but provides simulation results indicating the value of the final result. 
In this work, by adding another coordinate (the split number $\scip\in\gk{1,\ldots,d}$), we are able to reformulate the model as a Markovian branching process and prove the correct results.

The analysis of $\tf$-ary tree algorithms is mainly\footnote{However, a number of graph theoretic approaches have been carried out, see \cite{evseev2007interrelation} and references therein.} done by combining generating functions with tools from complex analysis, see \cite{massey1981collision,mathys1985q,fayolle1986functional,molle1992computation,yu2007high}
 for example. With the Markovian structure at hand, we can use the aforementioned tools to derive closed-form expressions for many observables of the process. 
 Arguably the most important characteristic of a tree algorithm is the\ac{CRI} length, denoted by $\rk{l_n}_{n\ge 0}$, conditioned on the number of packets in the initial collision $n$. 
 We analyse the law of $l_n$ by deriving a functional equation, which the moment generating function for $l_n$ solves, see Proposition \ref{PropositionFunctionalRelation}. To obtain an explicit formula for mean $L_n=\E\ek{l_n}$, we differentiate the moment generating function and solve the ensuing functional equation. This method also works for the variance of $l_n$, as well as for the higher moments. We also give the functional relations for the moment generating functions of the number of collisions and successes occurring during SICTA and derive the explicit formulas for their means.  
 We stress that our method works for a large class of observables, although the solution of the functional equations must be checked on a case-by-case basis, see the proof of Corollary \ref{CorollaryLNbinomialFormula}.
 
 Using the explicit formula for $L_n$, we leverage asymptotic analysis to extract the leading term. Contrary to many tree models (see \cite{drmota2009random}), the mean CRI length $L_n$ does not converge when divided by $n$ but instead has small, non-vanishing fluctuations. Asymptotic analysis was done in detail for \ac{STA} in the case of equal splitting in \cite{mathys1985q}, and for $\tf=2$ and biased splitting in \cite{fayolle1986functional}.
To derive the leading term from the explicit formulas for the mean throughput, number of collisions and successes, we developed a more robust expansion that works for both fair and biased splitting and any $\tf\ge 2$. We achieve this by using some explicit identities derived from the binomial series together with the  geometric sum formula. This can then be combined with the use of the residue theorem as in \cite{mathys1985q}. 
 
 We then calculate the extremal points for the leading order of the observables. We verify the conjecture that the maximal throughput of $\log(2)$ can be achieved for any $\tf$-ary SICTA, given suitable splitting probabilities. This conjecture was formulated in \cite{deshpande2022correction}, based on numerical simulations.

We also numerically simulate the minimal collision rate subject to a throughput constraint. As the number of collisions corresponds to the number of signals stored in the access point, this result helps to gauge memory requirements. We show that a small reduction in throughput allows for a (relatively) large reduction in collisions. This is of interest when the arrival rate is not too close to the critical, stability threshold, as one is able to reduce collisions without affecting the mean throughput.

In the final section of the paper, we give a recursive relation which allows for the calculation of the moment generating function of $l_n$ up to arbitrary degree, see Proposition \ref{PropositionMGF}. We also solve a functional equation for the mean delay for SICTA in steady state, see Proposition \ref{ProposotionDelayG}.

\section{Results}\label{sec:results}
\subsection{A mathematical model for SICTA}\label{subsec:math introd}
In this section, we give an abridged description of SICTA from the mathematical point of view, in order to be able to state the main results rigorously. For readers familiar with the SICTA, this can be skipped on first reading. For a full description of the algorithm we refer the reader to Section \ref{sec:backgound}.

The underlying objects of our study are $\tf$-ary ($d\ge 2$) labelled trees with random, integer valued labels. The label of the root is a fixed, non-random number in $\N_0$. Nodes with label $m>1$ have children $(c_1,\ldots,c_d)$. The labels of the children are distributed according to the multinomial distribution $\mathrm{Mult}(m,p)$, where $p\in [0,1]^d$ is a vector of splitting probabilities. The nodes with label $m\in\gk{0,1}$ have no children. Labelling only depends on the parent node, and hence the resulting tree has a Markovian structure. For the simple tree algorithm, the \ac{CRI} length $l_n$ is given by the total number of nodes in the tree. However, for SICTA certain nodes are skipped: if the sum of the labels to the left of a node is larger or equal to the label of the parent minus 1, this node will not be counted; see Section \ref{sec:backgound} for more details and a justification of this. We also refer the reader to Figure \ref{fig:sta_example_full} for an example. A definition of $l_n$ for SICTA is as follows: given a fixed node with label $n\ge 2$, we denote the last non-skipped slot by $\scip$, which is defined as
\begin{equation}
    \scip=\inf\gk{k\in\gk{1,\ldots, d}\colon \sum_{j=1}^k  {I_j}\ge n-1}\, ,
\end{equation}
where $\gk{I_j}_{j\in\gk{1,\ldots,d}}$ are the labels of the children of the node. As $\gk{I_j}_{j\in\gk{1,\ldots,d}}$ are multinomial distributed, their sum equals $n$ and hence $\scip$ is well-defined.

The recursive definition of the $l_n$ is then given by
\begin{equation}\label{Equationlnfirst}
    l_n=\begin{cases}
        1&\text{ if }n=0,1\, , \\
        \1\gk{\scip<d}+\sum_{j=1}^\scip l_{I_j}&\text{ if }n\ge 2\, .
    \end{cases}
\end{equation}
We refer the reader to Equation \eqref{EquationDefinitionCollision} and Equation \eqref{EquationSnBinon} for the formulas for the collisions and successes.
\subsection{Main results}
For $k\in\gk{0,\ldots,d-1}$, we write $\Fq(k)=\sum_{j=k+1}^dp_k$. 
\begin{theorem}\label{THM1}
For any $\tf\ge 2$ and any probability vector $p\in [0,1]^d$, we have
\begin{enumerate}
    \item For $n\ge 1$, we have that
\begin{equation}\label{EquationExactLn}
    \E\ek{l_n}=L_n=1+\sum_{i=2}^n\binom{n}{i}\frac{(-1)^{i}(i-1)\sum_{k=0}^{d-2}\Fq(k)^i}{1-\sum_{j=1}^dp_j^i}\, .
\end{equation}
\item 
We furthermore have that as $n\to\infty$
\begin{equation}
    \frac{L_n}{n}=\frac{\sum_{k=0}^{d-2}\Fq(k)}{-\sum_{j=1}^dp_j\log p_j}+g_1(n)+o(1)\, ,
\end{equation}
where $g_1(n)$ is given in Equation \eqref{TheRealEquationForGN}. Furthermore, if Equation \eqref{SinePertEq} has no solution, then $g_1(n)=0$.
\item The first term on the right-hand side of the previous equation is minimized for $p=\ps\in \R^d$ with $p_j=2^{-\min\gk{j,d-1}}$, $j\in\gk{1,\ldots,d}$. For this $p$, one has that
\begin{equation}\label{EquationAsymLN}
    \frac{\sum_{k=0}^{d-2}\Fq(k)}{-\sum_{j=1}^dp_j\log p_j}=\frac{1}{\log(2)}\, .
\end{equation}
Furthermore for $\ps$, $g_1(n)$ is bounded between $10^{-3}$ and $10^{-6}$. 
\end{enumerate}
\end{theorem}
The proof of the first statement is given in Corollary \ref{CorollaryLNbinomialFormula}, the one of the second statement in Proposition \ref{Proposition Asymptotic expansion} and the last statement in Lemma \ref{LemmaMinimalPoints}.

We summarize the results for other important observables of the SICTA process in Table \ref{tab:result_table}. In the interest of legibility, we have decided to postpone their definition until Section \ref{sec:analysis}. Proofs are very similar to the throughput, apart from the asymptotic relation for the number of successes, see Lemma \ref{LemmaSuccesses}.
\renewcommand{\arraystretch}{1.8}
\begin{table}[h]
    \centering
    \begin{tabular}{|c|c|c|c|c|}
    \hline
    Name & Symbol &Closed formula & Asymp. Leading Term & Evaluated at $\ps$\\
    \hline
    Throughput & $L_n$ & Equation \eqref{EquationExactLn} & $\frac{L_n}{n}\asymp\frac{\sum_{k=0}^{d-2}\Fq(k)}{-\sum_{j=1}^dp_j\log p_j}$& $\frac{1}{\log(2)}$\\
    \hline
    Collisions & $C_n$ & Equation \eqref{EquationClosedCollisions} & $\frac{C_n}{n}\asymp\frac{1-p_d}{-\sum_{j=1}p_j\log\rk{p_j}}$ & $\frac{1}{2\log(2)}$\\
    \hline
    Successes & $S_n$ & Equation \eqref{EquationSnBinon} & $\frac{S_n}{n}\asymp \frac{\sum_{k=2}^d p_k\log \Fq(k-1)}{\sum_{j=1}^d p_j\log p_j}$ & $\frac{1}{2}$\\
    \hline
    Idle slots & $I_n$ & Equation \eqref{ClosedEquaIn} & Equation \eqref{LeadingOrderIdle} & $\frac{1-\log(2)}{2\log(2)}$\\
    \hline
    \end{tabular}
    \caption{Summarizing the results for different observables of SICTA.}
    \label{tab:result_table}
\end{table}

Furthermore we obtain a number of results regarding the mean delay of SICTA in steady state. We state them in Section \ref{subsectionDelayAnalysis}, as they require a more technical background.

\section{Background}
\label{sec:backgound}

\subsection{Tree Algorithms}

The first tree-algorithm was introduced by Capetanakis in~\cite{capetanakis1979tree}.
It is also known as the Capetanakis-Tsybakov-Mikhailov type \ac{CRP}. 
The protocol addresses the classical \ac{RA} problem where several users must transmit packets to an \ac{AP} over a time-slotted shared multiple access channel with broadcast feedback. The most basic form of the algorithm is \ac{STA} which proceeds as follows. 
Assume that $\nuser$ packets are transmitted by $\nuser$ different users in a given slot, then: 
\begin{itemize}
    \item If $\nuser = 0$, then the slot is idle. 
    \item If $\nuser = 1$, then there is only one packet in the slot (also called a singleton), and the packet can be successfully decoded. 
    \item If $\nuser > 1$, the signals of $\nuser$ different transmissions interfere with each other, and no packet can be decoded. 
    This scenario is called a collision. 
    The users must re-transmit their packets according to the \ac{CRP}.
\end{itemize}


\textit{Collision Resolution Protocol} : At the end of every slot, the \ac{AP} broadcasts the outcome of the slot, i.e., idle (0), success (1), or collision (e) (where (e) stands for error), to all the users in the network. 
If the feedback is a collision, the $\nuser$ users independently split into $\tf$ groups. 
The probability that a user joins group $j$ is $p_j$ where $j \in \{1,...,\tf\}$, $\tf\ge 2$ and $p_j \in \ek{0,1}$. 
In the next slot, all the users who chose the first group ($j=1$) re-transmit their packets. If this results in a collision once again, then the process continues recursively. 
Users who have chosen $j>1$-st group observe the feedback. They wait until all users in $j-1$-st group successfully transmit their packets to the \ac{AP}. 
One can represent the progression of a \ac{CRI} in terms of $\tf$-ary trees as shown in Figure \ref{fig:sta_example_full}. Here, we show an example with the initial number of collided users, $\nuser=4$ and $\tf=3$. 
Each node on the tree represents a slot. The number inside the node shows the number of users that transmit in a given slot. 
The slot number is shown outside the node. 
After a collision node, the first group branches to the left of the tree, the second group branches in the middle, and the third group branches to the right. 
The number of slots needed from the first collision till the \ac{CRP} is complete is known as the \ac{CRI}. 


\begin{figure}
    \centering
    \includegraphics[width=0.35\textwidth]{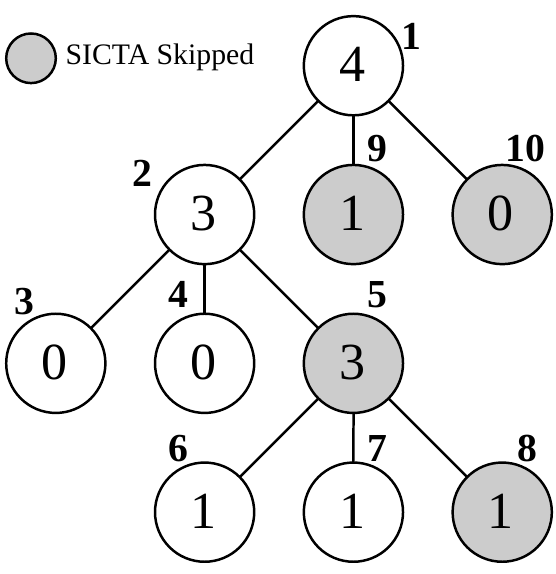}
    \caption{The tree illustration of the ternary ($\tf=3$) tree algorithm. The number inside the nodes in the tree represents the number of users transmitting in that slot. The number outside the node represents the slot number. Slots 5,8,9 and 10 will be skipped in the SICTA. }
    \label{fig:sta_example_full}
\end{figure}

The main performance parameter for tree algorithms is conditional throughput, defined as the ratio of users $\nuser$ and the total number of slots in a \ac{CRI} $L_n$. 
In the example from Figure \ref{fig:sta_example_full}, it is $0.4$ packets/slot. 
Furthermore, the asymptotic throughput (as $\nuser \rightarrow \infty$) is important to know the algorithm's \ac{MST}. The \ac{MST} gives the stability of the \ac{RA} scheme for a certain arrival rate of users, $\lambda\in \mathbb{R}^{+}$. 
For example, the stability of the \textit{gated} \ac{RA} scheme, \footnote{The stability is related to the MST for windowed access and free access with a different equation.} it is given by
\begin{equation}
    \text{MST} = \frac{1}{\lim_{\nuser \rightarrow \infty} \ecri_n / n} \geq \lambda\, .  
    \label{eq:proxy_mst}
\end{equation}
We will analyse the case of SICTA with a Poisson arrival rate of new packets in Section \ref{subsectionDelayAnalysis}.

\subsection{Successive Interference Cancellation}
\label{subsec:background_sic}

In \ac{STA} 
collision signals are discarded at the receiver. 
A new method where the receiver saves collision signals and tries to resolve more packets per slot was introduced in \cite{yu2007high}.
Here the receiver removes the signals of successful packets from the collision signals.
This process is known as \ac{SIC}. 
Let $Y_{s}$ be the signal of slot number $s$ and $X_{i}$ be the signal of the packet of user $i$. 
In the example from Figure \ref{fig:sta_example_full}, the receiver will save $Y_{1}$ and $Y_{2}$. 
In an interference-limited channel i.e. the one for which the noise can be effectively neglected, the received signal is the sum of all packets transmitted in that slot, $Y_{1} = X_{1} + X_{2} + X_{3} + X_{4}$ and $Y_{2} = X_{1} + X_{2} + X_{3}$. 
We keep the same slot indices as \ac{STA} for the legibility from the diagram. Since all the users are treated the same, we assume that the first user to be resolved is user 1, then user 2, and so on until user 4.  
In slot 6, the receiver gets $Y_{6} = X_{1}$. 
Since there is no interference in this slot, the receiver can decode the packet. 
It then is able to remove $X_{1}$ from $Y_{1}$ and $Y_{2}$.
Similarly, after slot 7 the receiver can remove $X_{2}$ from $Y_{1}$ and $Y_{2}$.
After removing $X_{1}$ and $X_{2}$, only $X_{3}$ remains in $Y_{2}$.
Thus the packet from user 3 can be decoded. 
The receiver can then proceed to remove $X_4$ from $Y_{1}$ and decode $X_{2}$ without the need for user 4 to have transmitted a packet after the first slot.
In this manner, the receiver can decode 2 packets after slot 7, resulting in a shorter \ac{CRI}.
One can easily see from the diagram that if all the signals from a particular node in the tree are removed, then all the remaining children of that node can be skipped. 
Another advantage of SICTA is the knowledge that the right-most branch of the tree can always be skipped. 
If the algorithm reaches the right-most branch of a node and still has not decoded all the signals from that node, this right-most branch will be a definite collision and can hence be skipped. 

The asymptotic throughput of SICTA was (incorrectly) shown to be $\frac{\ln \tf}{\tf-1}$ in \cite{yu2007high} achieved for fair-splitting.
Thus, \ac{SICTA} with fair splitting was thought to be the only configuration that achieves the optimal asymptotic throughput of $\ln{2}$ packets/slot.
However, a premise in their analysis for $\tf > 2$ was shown to be wrong in \cite{deshpande2022correction}.
In \cite{yu2007high}, it was assumed that only the right-most branch can be skipped. 
It fails to consider a scenario for $\tf > 2$, where more than one child node can be skipped when all the signals in the parent node are resolved.  
In the example from \ref{fig:sta_example_full}, \cite{yu2007high} failed to consider that slot 9 would be skipped after all the signals in $Y_{1}$ are decoded after slot 7. 
The correction paper \cite{deshpande2022correction}, did not provide the formal analysis but merely pointed out the mistake from \cite{yu2007high}. 
However, it did provide simulation results indicating that a special biased distribution of splitting probabilities, where
\begin{equation}
    p_{j} = \begin{cases}
    0.5^{j} & \; j \in \{1,.., \tf-1 \} \ \\
    0.5^{d-1} & \; j = d
    \end{cases}
\end{equation}
achieved an \ac{MST} of $\log(2)\approx0.693$ packets/slot for all values of $\tf$. 
In this work, we formally prove this indication to be correct. 
\subsection{Related work}\label{subsectionliterature}
As mentioned before, tree algorithms were introduced by \cite{capetanakis1979tree}.
A number of analytical results are due to Flajolet and Mathys, see \cite{mathys1984analysis,mathys1985q,fayolle1986functional}. Delay analysis was done in \cite{molle1992computation}. Yu and Giannakis introduced \ac{SICTA} in \cite{yu2007high}. There have been several publications regarding \ac{SICTA}, for example in \cite{andreev2011practical,peeters2015} variants of SICTA are considered and in \cite{stefanovic2020tree,stefanovic2021tree} the case where $K>1$ packets can be decoded in each step (multi-packet reception) is examined. The case of windowed and free access was studied in \cite{peeters2009}. Analysis of the depth of the resulting tree was carried out in \cite{janson1997analysis,holmgren2012novel}. Recently, large deviation analysis was applied to random access algorithms, see \cite{konig2022multi,konig2023throughput}. In these articles, the authors estimate the probability of rare events, such as large throughput deviations from their expected mean.


\section{Analysis}
\label{sec:analysis}
\subsection{Derivation of the functional equations}
Given a vector of probabilities $p=\rk{p_1,\ldots,p_d}$, at each collision each user independently chooses a slot $j\in\gk{1,\ldots,d}$ with probability $p_j$. Let $I_j$ denote the number of users who have chosen the $j$-th slot.

For a collision of $n$ packets, we recall that the last non-skipped slot $\scip$ is for SICTA defined as
\begin{equation}
    \scip=\inf\gk{k\in\gk{1,\ldots, d}\colon \sum_{j=1}^k  {I_j}\ge n-1}\, .
\end{equation}
The evolution of the \ac{CRI} length $l_n$ of $n$ collided users is then given by
\begin{equation}\label{Equationln}
    l_n=\begin{cases}
        1&\text{ if }n=0,1\, , \\
        \1\gk{\scip<d}+\sum_{j=1}^\scip l_{I_j}&\text{ if }n\ge 2\, .
    \end{cases}
\end{equation}
Since the remaining slots $\gk{{I_{\scip+1}},\ldots,{I_d}}$ hold at most one packet, they can be decoded from the original signal minus the decoded signals, see also \cite{deshpande2022correction}. However, the last slot can always be skipped, as it is the difference between the initial signal and the signals to the left.

Our first result is a functional equation for the moment generating function for $l_n$:
\begin{proposition}\label{PropositionFunctionalRelation}
Define for $x,z\in \C$
\begin{equation}
    \tmgfl(x,z)=\sum_{n\ge 0}\frac{x^n}{n!}\E\ek{z^{l_n}}\, .
\end{equation}
It then holds that
\begin{equation}\label{EquationRelQ}
    \tmgfl(x,z)=\prod_{i=1}^d \tmgfl(xp_i,z)+\sum_{k=0}^{d-2}(z-z^2)\rk{1+\Fq(k)x}\prod_{i=1}^k \tmgfl(xp_i,z)\, ,
\end{equation}
where $\Fq(k)=\sum_{j=k+1}^d p_k$.
\end{proposition}
In the literature, see for example \cite{yu2007high}, researchers work with $\mgfl(x,z)=\ex^{-x}\tmgfl(x,z)$. As it simplifies the notation, we work with $\tmgfl(x,z)$ for now and use $\mgfl(x,z)$ in the later parts of the article. From the proposition above, one can obtain closed formulas for the mean, variance and higher order terms. We apply this for the mean.
\begin{corollary}\label{CorollaryLNbinomialFormula}
We have that for all $n\ge 0$
\begin{equation}
        L_n=\E\ek{l_n}=1+\sum_{i=2}^n\binom{n}{i}\frac{(-1)^i(i-1)\sum_{k=0}^{d-2}\Fq(k)^i}{1-\sum_{j=1}^dp_j^i}\, .
\end{equation}
\end{corollary}
\begin{proof}
Note that by definition $\tmgfl(x,1)=\ex^x$. Set $K(x)= \frac{\d \tmgfl}{\d z}(x,1)$. By Equation \eqref{EquationRelQ}
\begin{equation}\label{EquationKx}
    K(x)=\sum_{i=1}^d K(p_ix)\ex^{(1-p_i)x}-\sum_{k=0}^{d-2}\rk{1+\Fq(k)x}\ex^{x\sum_{j=1}^kp_k}
\end{equation}
Set $\mgfl(x,z)=\ex^{-x}\tmgfl(x,z)$. Define the Poisson generating function $L(x)$ of $L_n$ as 
 \begin{equation}\label{EquationFormofLn}
     L(x)=\ex^{-x}\sum_{n\ge 0}\frac{x^n}{n!}L_n=\frac{\d Q}{\d z}(x,1)=\ex^{-x}K(x)\, .
 \end{equation}
Equation \eqref{EquationKx} now yields 
\begin{equation}
     L(x)=\sum_{i=1}^d L(p_ix)-\sum_{k=0}^{d-2}\rk{1+\Fq(k)x}\ex^{-\Fq(k)x}\, .
 \end{equation}
 If we use the expansion $L(x)=\sum_{n\ge 0}\alpha_n x^n$, this gives by comparing coefficients that for $n \ge 2$
 \begin{equation}
    \alpha_n=\alpha_n\sum_{i=1}^dp_i^n+\frac{1}{n!}{(-1)^n(n-1)\sum_{k=0}^{d-2}\Fq(k)^n}\, .
\end{equation}
Hence
 \begin{equation}\label{Equationalphan}
    \alpha_n=\frac{1}{n!}\frac{(-1)^n(n-1)\sum_{k=0}^{d-2}\Fq(k)^n}{1-\sum_{i=1}^dp_i^n}\, .
\end{equation}
Noting that by Equation \eqref{EquationFormofLn}
\begin{equation}
    L_n=\sum_{i=0}^n\frac{n!}{(n-i)!}\alpha_i\, .
\end{equation}
By Equation \eqref{Equationln}, we have $\alpha_0=1-\alpha_1=1$. Hence, we obtain
\begin{equation}\label{EquationLnBinom}
    \E\ek{l_n}=L_n=1+\sum_{i=2}^n\binom{n}{i}\frac{(-1)^i(i-1)\sum_{k=0}^{d-2}\Fq(k)^i}{1-\sum_{j=1}^dp_j^i}\, .
\end{equation}
Note that for $d=2$, we obtain the same result as \cite[Equation 30]{yu2007high}, as $\Fq(0)=1$. 
By taking higher-order derivatives (with respect to $z$) in Equation \eqref{EquationRelQ}, one can obtain closed formulas for the variance of $l_n$ as well as higher moments. We leave that to the reader. The proof of this corollary establishes the first claim in Theorem \ref{THM1}.
\end{proof}
We now proceed to the proof of Proposition \ref{PropositionFunctionalRelation}.

\textbf{Proof of Proposition \ref{PropositionFunctionalRelation}:}
The moment generating function $Q_n(z)$ is defined as
\begin{equation}
    Q_n(z)=\E\ek{z^{l_n}}=\sum_{k=1}^d\E\ek{z^{l_n}, \scip=k}\, ,\qquad\textnormal{for }z\in \C\, .
\end{equation}
Note that for $k<d$, we can split
\begin{equation}
    \E\ek{z^{l_n}, \scip=k}=\E\ek{z^{l_n}, \scip=k,\sum_{j=1}^\scip {I_j}=n-1}+\E\ek{z^{l_n}, \scip=k,\sum_{j=1}^\scip {I_j}=n}\, ,
\end{equation}
while for $k=d$, we have $\gk{\scip=k}=\gk{\scip=k,\sum_{j=1}^\scip {I_j}=n}$. In order to facilitate the analysis, we set
\begin{equation}
    \Pfrak_n\hk{d}=\gk{\mu\in \N^d_0\colon \sum_{k=1}^d\mu_k=n}\, ,\quad\text{and}\quad \binom{n}{\mu}=\binom{n}{\mu_1,\ldots,\mu_d}=\frac{n!}{\mu_1!\cdots\mu_d!}\, .
\end{equation}
Given probabilities $p=\rk{p_1,\ldots,p_d}$, we also introduce
\begin{equation}\label{Eqplambda}
    p(\mu)=\prod_{i=1}^d p_i^{\mu_i}\, ,\quad\text{for }\mu\in\N^d\, .
\end{equation}
For $k\le d$, we expand using Equation \eqref{Equationln}
\begin{equation}\label{equation7}
    \E\ek{z^{l_n}, \scip=k,\sum_{j=1}^\scip {I_j}=n}=\sum_{\mu\in\Pfrak_n\hk{k}}\binom{n}{\mu}p(\mu)\1\gk{\mu_k>1}z^{\1\gk{k<d}}\prod_{j=1}^kQ_{\mu_j}(z)\, .
\end{equation}
If $\scip=k$ and $\sum_{j=1}^k {I_j}=n$, then ${I_k}$ cannot be zero or one, because otherwise $\scip$ would have been bounded by $k-1$.

Note that $\1\gk{\mu_k>1}$ can be written as $1-\1\gk{\mu_k=0}-\1\gk{\mu_k=1}$ and that furthermore 
\begin{equation}
    \gk{\mu\in \Pfrak_{n}\hk{k}\colon \mu_k=0}\quad\text{is isomorphic to}\quad \Pfrak_{n}\hk{k-1}\, ,
\end{equation}
and similarly one has that $\gk{\mu\in \Pfrak_{n}\hk{k}\colon \mu_k=1}$ is isomorphic to $\Pfrak_{n-1}\hk{k-1}$. With this in mind, we rewrite the above equation as
\begin{multline}\label{FirstEqQn}
    \sum_{\mu\in\Pfrak_n\hk{k}}\binom{n}{\mu}p(\mu)\1\gk{\mu_k>1}z^{\1\gk{k<d}}\prod_{j=1}^kQ_{\mu_j}(z)=\sum_{\mu\in\Pfrak_n\hk{k}}\binom{n}{\mu}p(\mu)z^{\1\gk{k<d}}\prod_{j=1}^kQ_{\mu_j}(z)\\
    -\sum_{\mu\in\Pfrak_n\hk{k-1}}\binom{n}{\mu}p(\mu)z^{1+\1\gk{k<d}}\prod_{j=1}^{k-1}Q_{\mu_j}(z)-p_k\sum_{\mu\in\Pfrak_{n-1}\hk{k-1}}\binom{n}{\mu}p(\mu)z^{1+\1\gk{k<d}}\prod_{j=1}^{k-1}Q_{\mu_j}(z)\, .
\end{multline}
Note that the cases $\gk{\mu_k=0}$ and $\gk{\mu_k=1}$ give us an extra factor of $z$.

We now consider the case where $\sum_{j=1}^\scip {I_j}=n-1$, which implies $k<d$. Write $\Fp(k)=\sum_{j=1}^k p_k$ and $\Fq(k)=1-\Fp(k)$ for the cumulative distribution function induced by $p$. We obtain 
\begin{equation}\label{eq:first exps}
    \E\ek{z^{l_n}, \scip=k,\sum_{j=1}^\scip {I_j}=n-1}=nz\Fq(k)\sum_{\mu\in\Pfrak_{n-1}\hk{k}}\binom{n-1}{\mu}p(\mu)\1\gk{\mu_k>0}\prod_{j=1}^kQ_{\mu_j}(z)\, .
\end{equation}
Indeed, if $\sum_{j=1}^\scip {I_j}=n-1$, we have $n$ choices to select one packet and place it on the slots $\gk{\scip+1,\ldots, d}$. Summing over all possible slots gives us a probability of $p_{k+1}+\ldots+p_d=\Fq(k)$. We have $n-1$ packages left to distribute amongst the $k$ slots, which gives the multinomial coefficient.

In a similar fashion as before, we expand Equation \eqref{eq:first exps}
\begin{multline}\label{SecondEqQn}
   nz\Fq(k)\sum_{\mu\in\Pfrak_{n-1}\hk{k}}\binom{n-1}{\mu}p(\mu)\1\gk{\mu_k>0}\prod_{j=1}^kQ_{\mu_j}(z)\\=nz\Fq(k)\rk{\sum_{\mu\in\Pfrak_{n-1}\hk{k}}\binom{n-1}{\mu}p(\mu)\prod_{j=1}^kQ_{\mu_j}(z)
    -z\!\!\!\sum_{\mu\in\Pfrak_{n-1}\hk{k-1}}\!\binom{n-1}{\mu}p(\mu)\prod_{j=1}^{k-1}Q_{\mu_j}(z)} .
\end{multline}
From Equations \eqref{FirstEqQn} and \eqref{SecondEqQn} one could derive a recursive formula such as in \cite{massey1981collision}. 
However, it is not of much use to our analysis. 

On account of Equation \eqref{Equationln}, we have that
\begin{equation}
    \tmgfl(x,z)=(1+x)z+\sum_{k=1}^d \sum_{n\ge 2}\frac{x^n}{n!}E\ek{z^{l_n}, \scip=k}\, .
\end{equation}
We first examine the case $\scip=1$, as it is a bit different from the rest. We have that for $n\ge 2$
\begin{equation}\label{ZeroEqQnMisN}
    \E\ek{z^{l_n}, \scip=1}=p_1^nzQ_n(z)+np_1^{n-1}\Fq(1)zQ_{n-1}(z)\, ,
\end{equation}
as for $\scip=1$ either $n$ or $n-1$ packets must have picked the first slot. The first case has a probability of $p_1^n$ and the second of $np_1^{n-1}(1-p_1)$.

Recall that $p_1+\Fq(1)=1$. We hence get that
\begin{multline}
    \sum_{n\ge 2}\frac{x^n}{n!}\rk{p_1^nQ_n(z)+np_1^{n-1}\Fq(1)Q_{n-1}(z)}={\tmgfl(p_1x,z)-z-zp_1x+x\Fq(1)\ek{\tmgfl(p_1x,z)-z}}
   \\=\tmgfl(p_1x,z)\rk{1+\Fq(1)x}-z(1+x)\, ,
\end{multline}
by reindexing: consider the first summand in the above equation. By employing an index shift, we obtain
\begin{equation}
     \sum_{n\ge 2}\frac{x^n}{n!}p_1^nQ_n(z)=\rk{ \sum_{n\ge 0}\frac{(xp_1)^n}{n!}p_1^nQ_n(z)}-z-zp_1x\, .
\end{equation}
Fix now $k\ge 2$ and consider the case $\scip=k$. We begin with the case $\sum_{j=1}^k {I_j}=n$, which yields for $k\le d$
\begin{equation}
    \sum_{n\ge 2}\frac{x^n}{n!}E\ek{z^{l_n}, \scip=k,\sum_{j=1}^k {I_j}=n}=\sum_{n\ge 0}\frac{x^n}{n!}E\ek{z^{l_n}, \scip=k,\sum_{j=1}^k {I_j}=n}\, ,
\end{equation}
as for $\scip\ge 2$, one needs to have $n\ge 2$. By Equation \eqref{equation7} and Equation \eqref{FirstEqQn}
\begin{equation}\label{FirstEqQnMisN}
    \sum_{n\ge 2}\frac{x^n}{n!}E\ek{z^{l_n}, \scip=k,\sum_{j=1}^k {I_j}=n}=z^{\1\gk{k<d}}\prod_{j=1}^k\tmgfl(p_jx,z)-z^{1+\1\gk{k<d}}(1+p_kx)\prod_{j=1}^{k-1}\tmgfl(p_jx,z)\, ,
\end{equation}
which we illustrate with the last term in Equation \eqref{FirstEqQn}: recall that for non-negative sequences $\gk{a_n\hk{i}}_{i,n\ge 0}$, one has that
\begin{equation}\label{IllustrationEq}
    \prod_{i=1}^k\rk{\sum_{n\ge 0}a_n\hk{i}}=\sum_{n\ge 0}\sum_{\mu\in\Pfrak_n\hk{k}}\prod_{i=1}^k a_{\mu_i}\hk{i}\, .
\end{equation}
Hence, employing an index shift, we obtain
\begin{equation}
\begin{split}
    \sum_{n\ge 0}\frac{x^n}{n!}\sum_{\mu\in\Pfrak_{n-1}\hk{k-1}}\binom{n}{\mu}p(\mu)z^{1+\1\gk{k<d}}&\prod_{j=1}^{k-1}Q_{\mu_j}(z)\\=& xz^{1+\1\gk{k<d}}\sum_{n\ge 1}\sum_{\mu\in\Pfrak_{n-1}\hk{k-1}}\prod_{j=1}^{k-1}\frac{Q_{\mu_j}(z)(p_jx)^{\mu_j}}{\mu_j!}\\=&xz^{1+\1\gk{k<d}}\prod_{j=1}^{k-1}\rk{\sum_{n\ge 0}\frac{(p_jx)^n}{n!}Q_n(z)}\\
    =&xz^{1+\1\gk{k<d}}\prod_{j=1}^{k-1}Q(p_jx,z)\, .
    \end{split}
\end{equation}
The other terms in Equation \eqref{FirstEqQnMisN} follow similarly. Summing the right-hand side of Equation \eqref{FirstEqQnMisN} from $k=2$ to $d$ gives
\begin{equation}\label{EquationFirstOnesAdded}
    \prod_{j=1}^d\tmgfl(p_jx,z)-z(1+p_kx)\prod_{j=1}^{d-1}\tmgfl(p_jx,z)+\sum_{k=2}^{d-1} z\prod_{j=1}^k\tmgfl(p_jx,z)-z^{2}(1+p_kx)\prod_{j=1}^{k-1}\tmgfl(p_jx,z)\, .
\end{equation}

Using Equation \eqref{SecondEqQn}, the case $k<d$ and $\sum_{j=1}^k {I_j}=n-1$ gives in a similar fashion
\begin{equation}\label{SecondEqQnMisN}
    \sum_{n\ge 2}\frac{x^n}{n!}E\ek{z^{l_n}, \scip=k,\sum_{j=1}^k {I_j}=n-1}=xz\Fq(k)\prod_{j=1}^k\tmgfl(p_jx,z)-xz^2\Fq(k)\prod_{j=1}^{k-1}\tmgfl(p_jx,z)\, .
\end{equation}
Summing the right-hand side of the above over all $k\in\gk{2,\ldots,d-1}$ gives
\begin{equation}\label{EquationAddingSecond}
    \sum_{k=2}^{d-1}z\Fq(k)\prod_{j=1}^k\tmgfl(p_jx,z)-xz^2\Fq(k)\prod_{j=1}^{k-1}\tmgfl(p_jx,z)\, .
\end{equation}
When adding Equations \eqref{ZeroEqQnMisN}, Equation \eqref{EquationFirstOnesAdded} and Equation \eqref{EquationAddingSecond}, we notice that the $d-1$ term cancels. Hence, we obtain the functional relation
 \begin{equation}
     \tmgfl(x,z)=\prod_{i=1}^d \tmgfl(xp_i,z)+\sum_{k=0}^{d-2}(z-z^2)\rk{1+\Fq(k)x}\prod_{i=1}^k \tmgfl(xp_i,z)\, .
 \end{equation}
This concludes the proof of Proposition \ref{PropositionFunctionalRelation}.

One can also look at the number of collisions $c_n$, which follow the recursive equations
\begin{equation}\label{EquationDefinitionCollision}
    c_n=\begin{cases}
        0&\text{ if }n\in\gk{0,1}\, ,\\
        \1\gk{M<d}+\sum_{j=1}^M c_{I_j}&\text{ if }n\ge 2\, .
    \end{cases}
\end{equation}
In this case, following similar steps as for the throughput, we obtain
\begin{equation}\label{EquationClosedCollisions}
    C_n=\sum_{i=2}^n\binom{n}{i}\frac{(-1)^i\rk{i-1}\rk{1-p_d^i}}{1-\sum_{j=1}^d p_j^i}\, .
\end{equation}
The result relies on the functional relation for the exponential moment generating function
\begin{equation}
    \widetilde{R}(x,z)=\sum_{n\ge 0}\frac{x^n}{n!}\E\ek{z^{c_n}}\, ,
\end{equation}
given by
\begin{equation}
    \widetilde{R}(x,z)=(1+x)(1-z)+(z-1)\rk{xp_d+1}\prod_{i=1}^{d-1}\widetilde{R}(p_ix,z)+\prod_{j=1}^d \widetilde{R}(p_jx,z)\, .
\end{equation}

We also give the formula for $S_n$, the expected number of successes
\begin{equation}\label{EquationSnBinon}
    S_n=1+\sum_{i=2}^n\binom{n}{i}\frac{(-1)^{i-1}i\rk{1-\sum_{k=1}^{d}p_{k}\Fq(k-1)^{i-1}}}{{1-\sum_{j=1}^dp_j^i}}\, ,
\end{equation}
based on
\begin{equation}
    s_n=\begin{cases}
        0&\text{ if }n=0\, ,\\
        1&\text{ if }n=1\, ,\\
        \sum_{j=1}^Ms_{I_j}&\text{ if }n\ge 2\, .
    \end{cases}
\end{equation}
Its exponential moment generating function $ \widetilde{S}(x,z)$ satisfies
\begin{equation}
     \widetilde{S}(x,z)=(1-p_1)x(z-1)+\prod_{i=1}^{d}\widetilde{S}(xp_i,z)+\sum_{k=1}^{d-1} xp_{k+1}\rk{1-z}\prod_{i=1}^{k}\widetilde{S}(xp_i,z)\, .
\end{equation}
Finally, the number of idle slots 
\begin{equation}
    i_n=\begin{cases}
    1&\text{ if }n=0\, ,\\
    0&\text{ if }n=1\, ,\\
    \sum_{j=1}^Mi_{I_j}&\text{ if }n\ge 2\, ,
    \end{cases}
\end{equation}
satisfies $i_n=l_n-c_n-s_n$ and hence satisfies
\begin{equation}\label{ClosedEquaIn}
    I_n=\sum_{i=2}^n\binom{n}{i}\frac{(-1)^{i}(i-1)\rk{\sum_{k=1}^{d-2}\Fq(k)^i+p_d^i+\frac{i}{i-1}\sum_{k=1}^d p_k\Fq(k-1)^{i-1}}}{\rk{1-\sum_{j=1}^dp_j^i}}\, .
\end{equation}
Note that the above procedure can be carried out for any observable as long as it is additive as we move down the tree. This is a common occurrence in the analysis of random trees, see \cite{drmota2009random}. A further example of such an additive observable would be the number of nodes with degree $R$ or larger than $R$. This might be of interest in practice as an interference of many signals could be difficult to control in terms of noise.
 \subsection{Asymptotic Analysis}
 \label{SubsectionAsymptoticAnalysis}
We extend the methods from \cite{mathys1985q} to allow for asymptotic analysis both in the equal-split case as well as in the biased case. The first key identity for our method is
\begin{equation}\label{EquationGeometric}
    \frac{1}{1-\sum_{j=1}^d x_j}=\sum_{m\ge 0}\sum_{\mu\in\Pfrak_m^d}\binom{m}{\mu}\prod_{j=1}^d x_j^{\mu_j}\, ,
\end{equation}
which follows from the geometric sum and expanding $\rk{\sum_{j=1}^d x_j}^m$. The other identity is 
\begin{equation}\label{EquationBinom}
    \sum_{i=1}^n \binom{n}{i}(-1)^n(i-1)x^i=\sum_{i=2}^n \binom{n}{i}(-1)^n(i-1)x^i=1-(1-x)^{n-1}\rk{1+(n-1)x}\, ,
\end{equation}
which follows from the binomial theorem through differentiation, as
\begin{equation}
\rk{1+x}^{n-1}nx=\sum_{i=1}^n\binom{n}{i}ix^i    \qquad\text{and}\qquad\sum_{i=1}^n\binom{n}{i}x^i=\rk{1+x}^n-1\, .
\end{equation}
We now state the main result of this section.
\begin{proposition}\label{Proposition Asymptotic expansion}
\begin{enumerate}
    \item If the equation
\begin{equation}\label{SinePertEq}
    p_1^{1/k_1}=\ldots=p_d^{1/k_d}\, ,
\end{equation}
has no positive integer solution, then 
\begin{equation}
    \frac{L_n}{n}=\frac{\sum_{k=0}^{d-2}\Fq(k)}{-\sum_{j=1}^d p_j\log p_j}+o\rk{1}\, .
\end{equation}
\item If Equation \eqref{SinePertEq} does have a positive integer solution, then
\begin{equation}\label{EquationAsympFormLm}
    \frac{L_n}{n}=\frac{\sum_{k=0}^{d-2}\Fq(k)}{-\sum_{j=1}^d p_j\log p_j}+g_1(n)+o\rk{1}\, ,
\end{equation}
where $g_1(n)$ is given in Equation \eqref{TheRealEquationForGN}.
\end{enumerate}
\end{proposition}
Note that the above proposition establishes the second claim in the proof of Theorem \ref{THM1}.
\begin{proof}
Recall that due to Equation \eqref{EquationLnBinom}
\begin{equation}
    {L_n}=1+ \sum_{k=0}^{d-2}\sum_{i=2}^n\binom{n}{i}\frac{(-1)^i(i-1)\Fq(k)^i}{1-\sum_{j=1}^dp_j^i}\, .
\end{equation}
As the expression on the right-hand side is additive in $k$, it suffices to calculate the asymptotic behavior for $k\in \gk{0,\ldots, d-2}$ fixed. Hence our goal is to calculate the asymptotic value of
\begin{equation}\label{Equation21523}
    \frac{1}{n}\sum_{i=2}^n\binom{n}{i}\frac{(-1)^i(i-1)\alpha^i}{1-\sum_{j=1}^dp_j^i}\qquad\text{for}\qquad\alpha=\Fq(k)\, .
\end{equation}
We first apply Equation \eqref{EquationGeometric} to rewrite the above as
\begin{equation}
    \sum_{i=2}^n\binom{n}{i}\frac{(-1)^i(i-1)\alpha^i}{1-\sum_{j=1}^dp_j^i}=\sum_{m\ge 0}\sum_{\mu\in\Pfrak_m^d}\binom{m}{\mu}\sum_{i=2}^n\binom{n}{i}(-1)^i(i-1)\alpha^ip(\mu)^i\, ,
\end{equation}
where $p(\mu)$ was defined in Equation \eqref{Eqplambda}. Now we can apply Equation \eqref{EquationBinom} to eliminate the sum over $i$:
\begin{equation}
   L_n\sim  \sum_{m\ge 0}\sum_{\mu\in\Pfrak_m^d}\binom{m}{\mu}\rk{1-\rk{1-\alpha p(\mu)}^{n-1}\ek{1+(n-1)p(\mu)\alpha}}\, .
\end{equation}
In order to increase legibility, we will switch from $n-1$ to $n$. We use the approximation $\rk{1-x}^n=\ex^{-xn+\Ocal\rk{x^2n}}$ which yields
\begin{equation}
   {L_{n+1}}\sim\sum_{m\ge 0}\sum_{\mu\in\Pfrak_m^d}\binom{m}{\mu}\rk{1-\ex^{-\alpha n p(\mu)}\ek{1+\alpha np(\mu)}}\, .
\end{equation}
See \cite{mathys1985q} for the argument why the $\ex^{\Ocal\rk{x^2n}}$ term is negligible. We now write the above as
\begin{equation}
    {L_n}\sim\sum_{m\ge 0}\sum_{\mu\in\Pfrak_m^d}\binom{m}{\mu}f(\alpha n p(\mu))\quad\text{where}\quad f(x)=1-\ex^{-x}(1+x)\, .
\end{equation}

Recall that for a function $f\colon \C\to\C$, its \textit{Mellin} transform is given by 
\begin{equation}\label{EquationMellin}
    \Mcal\ek{f;s}=\int_0^\infty x^{s-1}f(x)\d x\quad\text{ and its inverse }\quad f(x)=\frac{1}{2\pi i}\int_{c-i\infty}^{c+i\infty}x^{-s}\Mcal\ek{f;s}(s)\d s\, ,
\end{equation}
for suitable $c\in\R$, see \cite{FLAJOLET19953} for a reference. For $f(x)=1-\ex^{-x}(1+x)$, one has that $f(x)=\Ocal\rk{x^2}$, as $x\to 0$. One furthermore has that $f'(x)=x\ex^{-x}$. Hence, using integration by parts in Equation \eqref{EquationMellin}, we obtain
\begin{equation}
    \Mcal\ek{f;s}=-\frac{1}{s}\int_0^\infty x^{s}f'(x)\d x=-\frac{1}{s}\int_0^\infty x^{s+1}\ex^{-x}\d x=-\frac{\Gamma(s+2)}{s} =-(s+1)\Gamma(s)\, ,
\end{equation}
as long as $0>\Re(s)>-2$. Here $\Gamma$ denotes the (complex) gamma function which satisfies $\Gamma(s+1)=s\Gamma(s)$. Using the above, we get that
\begin{multline}
    L_n\sim \frac{-1}{2\pi i}\sum_{m\ge 0}\sum_{\mu\in\Pfrak_m^d}\binom{m}{\mu}\int_{3/2-i\infty}^{3/2+i\infty }{p(\mu)^{-s}\alpha^{-s}n^{-s}(s+1)\Gamma(s)}\d s\\
    =\frac{-1}{2\pi i}\int_{3/2-i\infty}^{3/2+i\infty }\frac{\alpha^{-s}n^{-s}(s+1)\Gamma(s)}{1-\sum_{j=1}^d p_j^{-s}}\d s\, ,
\end{multline}
where we used (see Equation \eqref{IllustrationEq})
\begin{equation}
    \sum_{m\ge 0}\sum_{\mu\in\Pfrak_m^d}\binom{m}{\mu}p(\mu)^{-s}= \sum_{m\ge 0}\sum_{\mu\in\Pfrak_m^d}\binom{m}{\mu}\prod_{i=1}^d\rk{p_i^{-s}}^{\mu_i}=\frac{1}{1-\sum_{j=1}^d p_j^{-s}}\, ,
\end{equation}
in the last step. Note that at $s=-1$, the integrand has a pole of first order. We expand
\begin{equation}
    1-\sum_{j=1}^d p_j^{-s}=1-\sum_{j=1}^d p_j \ex^{-(s+1)\log p_j}\sim 1-\sum_{j=1}^d p_j\rk{1-(s+1)\log p_j}=(s+1)\sum_{j=1}^d p_j\log p_j\, ,
\end{equation}
as $s\to -1$. Hence, the residue of the integrand is given by
\begin{equation}
    \mathrm{Res}\rk{\frac{\alpha^{-s}n^{-s}(s+1)\Gamma(s)}{1-\sum_{j=1}^d p_j^{-s}};s=-1}=n\frac{\alpha}{-\sum_{j=1}^d p_j\log p_j}\, ,
\end{equation}
where we used $(s+1)\Gamma(s)\sim -1$ as $s\to -1$. If the equation
\begin{equation}
    p_1^{1/k_1}=\ldots=p_d^{1/k_d}\, ,
\end{equation}
has no positive integer solution, $s=-1$ is the only pole (see \cite[Equation 3.67]{mathys1985q}) and hence the residue theorem gives
\begin{equation}
    \frac{-1}{2\pi i}\int_{3/2-i\infty}^{3/2+i\infty }\frac{\alpha^{-s}n^{-s}(s+1)\Gamma(s)}{1-\sum_{j=1}^d p_j^{-s}}\d s=n\frac{\alpha}{-\sum_{j=1}^d p_j\log p_j}\, ,
\end{equation}
see \cite{mathys1985q} for a description of the contour and further details.

If there are positive integer solutions to Equation \eqref{SinePertEq}, then there are infinitely many poles, this was shown in \cite[Equation 3.67]{mathys1985q}. We abbreviate $c=p_1^{1/k_1}$. The set of poles other than $s=-1$ is given by the set $S$, where
\begin{equation}
    S=\gk{(x,y)\text{ with }x\le -1, y\in\R\setminus\gk{0}\colon \sum_{j=1}^d p_j^{-x}\ex^{iy\log p_j}=1}\, ,
\end{equation}
see \cite[Equation 3.66]{mathys1985q}. We then have that
\begin{equation}
    \frac{-1}{2\pi i}\int_{3/2-i\infty}^{3/2+i\infty }\frac{\alpha^{-s}n^{-s}(s+1)\Gamma(s)}{1-\sum_{j=1}^d p_j^{-s}}\d s=n\frac{\alpha}{-\sum_{j=1}^d p_j\log p_j}+nf_1(n,\alpha)\, ,
\end{equation}
where the reside theorem again gives that
\begin{equation}
    f_1(n,\alpha)=\sum_{(x,y)\in S}\frac{\alpha^{x-iy}n^{x-iy}\Gamma\rk{x-iy}\rk{x-iy+1}}{-\sum_{j=1}^d p_j\log p_j}\, .
\end{equation}
For a derivation and the convergence of the sum, we refer the reader to \cite{mathys1985q}. 

By substituting the $\alpha=\Fq(k)$ and taking the sum over $k$ in Equation \eqref{Equation21523}, we hence get that
\begin{equation}
    \frac{L_n}{n}=\frac{\sum_{k=0}^{d-2}\Fq(k)}{-\sum_{j=1}^d p_j\log p_j}+g_1(n)+o\rk{1}\, ,
\end{equation}
with
\begin{equation}\label{TheRealEquationForGN}
    g_1(n)=\sum_{k=0}^{d-2}f_1\rk{n,\Fq(k)}
\end{equation}
For the bounds on $g_1(n)$, we refer the reader to \cite[Table 1]{mathys1985q}.
\end{proof}
Similarly, one finds that
\begin{equation}\label{EquationAsymptCN}
    \frac{C_n}{n}= \frac{1-p_d}{-\sum_{j=1}p_j\log\rk{p_j}}+g_2(n)+o\rk{1}\, ,
\end{equation}
where $g_2(n)=f_1\rk{n,1-p_d}$.
 The obtain the asymptotic number of successes, more work is needed. We sketch the main steps and leave the rest to the reader. 
 \begin{lemma}\label{LemmaSuccesses}
 For $d\ge 2$ 
 \begin{equation}\label{EquationNumberOfSuccesses}
     \frac{S_n}{n}=  \frac{\sum_{k=2}^d p_k\log \Fq(k-1)}{\sum_{j=1}^d p_j\log p_j}+g_3(n)+o\rk{1}\quad\text{as}\quad n\to \infty\, ,
 \end{equation}
 where $g_3$ is given in Equation \eqref{Equation for g4}.
 \end{lemma}
 \begin{proof}
 Starting from Equation \eqref{EquationSnBinon}, we express
 \begin{equation}
     p_{k}\Fq(k-1)^{i-1}=\frac{p_k}{\Fq(k)}\Fq(k-1)^i=\frac{p_k}{\Fq(k)}q_k^i\, ,
 \end{equation}
 where we write $\Fq(k-1)=q_k$ to keep the ensuing formulas shorter.
 
 Using the expansion as in the proof of Proposition \ref{Proposition Asymptotic expansion}, we get that
 \begin{equation}
     S_n\sim \sum_{m\ge 0}\sum_{\mu\in\Pfrak_m^d}\binom{m}{\mu}f_2( n p(\mu))\quad\text{where}\quad f_2(x)=\sum_{k=1}^d p_kx\rk{\ex^{-x}-\ex^{-xq_k} }\, .
\end{equation}
Note that
\begin{equation}
    f_2(x)\sim x^2\sum_{k=2}^d p_k \Fp(k-1)\quad\text{as}\quad x\to 0\, .
\end{equation}
Here, recall that $\Fp(i)=\sum_{j=1}^i p_j$. The above expansion gives that for $\Re(s)>-2$, the Mellin transform is well defined
\begin{equation}
    \Mcal\ek{f_2;s}=\Gamma(s+1)\sum_{k=2}^d p_k\rk{1-q_k^{-1-s}}\, .
\end{equation}
Furthermore, $\Mcal\ek{f_2;s}$ has a removable singularity at $s=-1$:
\begin{equation}
    \Mcal\ek{f_2;s}\sim -\sum_{k=2}^d p_k\log \Fq(k-1)\quad\text{as }s\to -1\, ,
\end{equation}
where we used that $\Gamma(s)\sim s^{-1}$ as $s\to 0$ as well as $\Fq(k-1)=q_k$. Hence, using Equation \eqref{EquationMellin}, we get that
\begin{equation}
    \sum_{m\ge 0}\sum_{\mu\in\Pfrak_m^d}\binom{m}{\mu}f_2( n p(\mu))=\frac{1}{2\pi i}\int_{-3/2-i\infty}^{-3/2+i\infty}\frac{n^{-s}\Gamma(s+1)\sum_{k=2}^d p_k\rk{1-q_k^{-1-s}}}{1-\sum_{j=1}^d p_j^{-s}}\, .
\end{equation}
From this, one obtains using the residue theorem as in the proof of Proposition \ref{Proposition Asymptotic expansion}
\begin{equation}
    \frac{S_n}{n}= \frac{\sum_{k=2}^d p_k\log \Fq(k-1)}{\sum_{j=1}^d p_j\log p_j}+g_3(n)+o\rk{1}\, ,
\end{equation}
where 
\begin{equation}\label{Equation for g4}
    g_3(n)=\sum_{y\in S}\frac{\Mcal\ek{f_2; -1+iy}}{-\sum_{j=1}^d p_j\log p_j}\, ,
\end{equation}
unless Equation \eqref{SinePertEq} has no positive integer solution, in which case $g_3$ is equal to zero.
 \end{proof}
We can use the results for $L_n,C_n,S_n$ to obtain that
\begin{equation}\label{LeadingOrderIdle}
    \frac{I_n}{n}=\frac{\sum_{k=1}^{d-2}\Fq(k)+p_d+\sum_{k=2}^d p_k\log \Fq(k-1)}{-\sum_{j=1}^d p_j\log p_j}+g_4(n)+o\rk{1}\, ,
\end{equation}
where $g_4(n)=g_1(n)-g_2(n)-g_3(n)$, see Equation \eqref{TheRealEquationForGN}, Equation \eqref{Equation for g4} and after Equation \eqref{EquationAsymptCN}.
 \subsection{Minimization}
 \label{SubsectionMinimization}
 In this section, we calculate the values of $p$ which maximize throughput and success-rate, and minimize collisions and skipped slots. 
 
 Recall that $\Fq(k)=1-\sum_{j=1}^k p_j$. To achieve the maximum throughput, we want to minimize the main term in Equation \eqref{EquationAsympFormLm}, i.e.,
 \begin{equation}
     \frac{\sum_{k=0}^{d-2}\Fq(k)}{-\sum_{j=1}^d p_j\log p_j}\, .
 \end{equation}
 We do this in the next lemma:
 \begin{lemma}\label{LemmaMinimalPoints}
 For $\ps=\ps_d$, given by $\ps_i=2^{-\min\{i,d-1\}}$,
 the term
 \begin{equation}
      \frac{\sum_{k=0}^{d-2}\Fq(k)}{-\sum_{j=1}^d p_j\log p_j}\, ,
 \end{equation}
 is minimized. Furthermore, at $\ps$, we have that
 \begin{equation}
      \frac{\sum_{k=0}^{d-2}\Fq(k)}{-\sum_{j=1}^d p_j\log p_j}\Big|_{p=\ps}=\frac{1}{\log(2)}\, .
 \end{equation}
 No other minima exist besides $\ps$.
 \end{lemma}
 Note that the above lemma establishes the final claim in Theorem \ref{THM1} and also confirms the prediction from \cite{deshpande2022correction}.
\begin{proof}
We abbreviate
\begin{equation}
       \frac{\sum_{k=0}^{d-2}\Fq(k)}{-\sum_{j=1}^d p_j\log p_j}=\frac{N}{D}\, ,
\end{equation}
in order to keep the ensuing equations shorter. Suppose that $\mu$ is our Lagrange parameter, we obtain that for $i\le d-2$
\begin{equation}
    \mu=\frac{\d }{\d p_i}\frac{N}{D}=\frac{D(d-1-i)+N\rk{1+\log p_i}}{D^2}\, ,
\end{equation}
as the parameter $p_i$ appears $(d-1-i)$--times in the sum $\sum_{k=0}^{d-2}\Fq(k)$. However, the parameters $p_{d-1}$ and $p_d$ do not appear in the numerator and hence for $i=d-1,d$
\begin{equation}
    \mu=\frac{\d }{\d p_i}\frac{N}{D}=\frac{N\rk{1+\log p_i}}{D^2}\, .
\end{equation}
This implies that $p_{d-1}=p_d$. Using the two equations for $\mu$ and multiplying by $D^2$, shows that for $1\le i<j<d$
\begin{equation}
    D(j-i)=N\log \frac{p_j}{p_i}\, .
\end{equation}
Hence, for the above choice of $i,j$
\begin{equation}
   \frac{N}{D}= \frac{j-i}{\log\frac{p_j}{p_i}}\, .
\end{equation}
Choosing $j=i+1$, we obtain that for some $c>0$
\begin{equation}
    \frac{p_{i+1}}{p_i}=c\quad\text{ for all }i<d-1\text{, and hence }p_i=2^{-i}\, ,\quad\text{ for all }i<d\, .
\end{equation}
Write $\ps=\ps_d$ for the above distribution, given by $\ps_i=2^{-\min\{i,d-1\}}$,
\begin{equation}
    \ps_2=\rk{\frac{1}{2},\frac{1}{2}}\qquad\text{and}\qquad \ps_4=\rk{\frac{1}{2},\frac{1}{4},\frac{1}{8},\frac{1}{8}}\, ,
\end{equation}
for example.

We have that for $p=\ps$
\begin{equation}
    N=\sum_{k=0}^{d-2}\Fq(k)=\sum_{k=0}^{d-2}2^{-k}=2-2^{-d+2}\, .
\end{equation}
For the denominator, one obtains
\begin{equation}
    D=-\sum_{j=1}^d p_j\log p_j=\log(2)\rk{\sum_{j=1}^{d-1}j2^{-j}+(d-1)2^{-d+1}}=\log(2)\rk{2-2^{-d+2}}\, ,
\end{equation}
where we have used the finite geometric sum formula in the last step. The two equations above imply that for our throughput maximizing distribution, one obtains that
\begin{equation}
     \frac{L_n}{n}=\frac{1}{\log(2)}+g_1(n)+o\rk{n^{-1}}\, ,
\end{equation}
This concludes the proof.

Alternatively, one can use the following inductive argument why for $\ps$, $L_n$ remains constant as $d\ge 2$ varies: for $d=3$, one can combine the two $1/4$ weighted branches into one. As the $1/2$ weighted branch has the same law as the one for $d=2$ and $L_n$ is additive in the branches, this shows that the $L_n$ is the same for $d=2$ and $d=3$, given $\ps$ as splitting probability. One can then inductively carry this over to higher $d$'s. 
\end{proof}
For $\ps$, we also obtain
\begin{equation}
    \frac{C_n}{n}\sim \frac{1}{2\log(2)}+g_2(n)\qquad\text{and}\qquad\frac{S_n}{n}\sim \frac{1}{2}+g_3(n)\, ,
\end{equation}
independent of $d$, the cardinality of the split.

Note that we can minimize $C_n$ by setting $p_i=0$ for $i<d$ and $p_d=1$. However, this is not a sensible choice, as the algorithm will never terminate. Furthermore, the first term in Equation \eqref{EquationNumberOfSuccesses} is maximized for the same choice of $p$.
\subsection{Collisions vs throughput}
\label{SubsectionCollisionsVsThroughput}
As we have seen in the previous section, the throughput-maximizing distribution does not minimize collisions for SICTA. We show that a small reduction in throughput can lead to a large reduction in the number of collisions. For example, a 20\% reduction in throughput allows for a 39\% reduction of occurring collisions, from $0.72$ collisions per package down to $0.44$ collisions per package. Below, we have plotted the minimal achievable collision rate, given a throughput reduction of at most $x$-percent, where $x$ ranges from $0\%$ to $20\%$. 
\begin{figure}[h]
    \centering
    \includegraphics[width=0.75\linewidth]{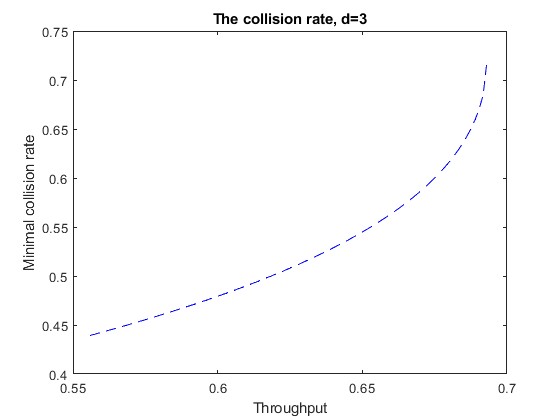}
    \caption{The minimal obtainable collision rate, constrained by achieving a certain throughput rate. The figure was obtained numerically using a standard solver for constraint non-linear optimisation problems. $\ps$ was used as a the initial value.}
    \label{figcollisionrate}
\end{figure}
The graph in Figure \ref{figcollisionrate} does not change as we vary the number of branches $d$. The corresponding probability distributions $p\in [0,1]^d$ can be obtained by using a standard constrained nonlinear multivariable solver, using $\ps$ as initial value.

\subsection{Delay Analysis}
\label{subsectionDelayAnalysis}
In this section, we look at SICTA with gated access. We give recursive formulas which allow for an approximation of the mean delay as well as transition matrix of the \ac{CRI} lengths.

We now assume that packets arrive at random times with an arrival rate $\l>0$. Packets wait and accumulate until the algorithm has resolved the previous collision. We define $\gk{c_{k+1}}_{k=0}^\infty$ the (random) sequence where $c_k$ is the length of the $k$-th \ac{CRI} assuming a Poisson arrival rate $\l>0$ of new packets. Its randomness is twofold, once from the \ac{CRI} itself, but also from the Poisson arrival of new packages. Let $\gk{s_{k+1}}_{k=0}^\infty$ be the number of packages arriving during the $k$-th \ac{CRI}. If we condition on $c_k=i$, $s_{k+1}$ is $\poi(\l i)$ distributed. Hence, we have that $\gk{c_{k+1}}_{k=0}^\infty$ is a Markov chain. Let $\pi=\gk{\pi_i}_i$ be the invariant distribution, which exists for $\l<\mathrm{MST}$, see Equation \eqref{eq:proxy_mst} or \cite{molle1992computation}. The probability that a tagged packet joins the system during a \ac{CRI} of length $n$ is given by
 \begin{equation}
     \wpi_n=\frac{n\pi_n}{\sum_{j=1}^\infty j\pi_j }\, ,
 \end{equation}
 see also \cite{molle1992computation}.
 
 Let $t=t_0+t_2$ the total delay of a given packet, made up from waiting $t_0$ slots for the previous \ac{CRI} to finish and then the time in the algorithm itself, denoted by $t_2$. Note that $t_0\perp t_2$ and that $t_0\sim \Ucal(0,n)$. The distribution of $t_2$ is given by $l_{1+R_n}$, where $R_n=\poi(\l n)$.
 \subsubsection{Steady-State Distribution of the \ac{CRI} Length.}
In this section, we state a functional recursive relation which allows for the computation of the moment generating function $\mgfl(x,z)$ up to arbitrary order. This recursive relation also allows for an asymptotic computation of the transition matrix $P_{i,j}=\P\rk{c_2=j|c_1=i}$.

\begin{proposition}\label{PropositionMGF}
Recall the moment generating function $Q(x,z)=\ex^{-x}\sum_{n\ge 0}x^n \E\ek{z^{l_n}}/n!$, used for the computation of moments of $l_n$. Write
\begin{equation}
    Q(x,z)=\sum_{j\ge 0}z^j q_j(x)\, ,
\end{equation}
where 
\begin{equation}
    q_j(x)=\sum_{n=0}^\infty \P\rk{l_n=j}\ex^{-x}\frac{x^n}{n!}\, .
\end{equation}
Then, there exists a recursive equation which for every $j\ge 1$ gives $q_j(z)$ in terms of $\gk{q_i(z)}_{i=0}^{j-1}$, see Equation \eqref{EquationSystemOfEq}. Furthermore, $q_0(z)=0$.
\end{proposition}
Before embarking on a proof of the above proposition, we show how it enables us to calculate the transition matrix of the \ac{CRI} lenths:

Recall that the arrival rate is $\l>0$ and that new packets arrive according to a Poisson process with parameter $\l>0$.
\begin{corollary}
The probability at steady-state to observe a \ac{CRI} length of $j$ after having observed a \ac{CRI} length of  of $i$ is given by
\begin{equation}
    P_{i,j}=q_j(\l i)\, .
\end{equation}
\end{corollary}
Indeed, note that by doing a case distinction
\begin{equation}
    P_{i,j}=\sum_{n=0}^\infty \P\rk{s_{k+1}=n|c_k=i}\P\rk{l_n=j}=\sum_{n=0}^\infty \P\rk{l_n=j}\ex^{-\l i}\frac{\rk{\l i}^n}{n!}=q_j(\l i)\, ,
\end{equation}
where $q_j(x)$ was given in the proposition above. We now prove Proposition \ref{PropositionMGF}.

\textbf{Proof of Proposition \ref{PropositionMGF}:}
Equation \eqref{Equationln} gives
\begin{equation}\label{Initial conditions q_j}
    q_j(x)=\begin{cases}
        0&\text{ if }j=0\, ,\\
        (1+x)\ex^{-x}&\text{ if } j=1\, ,
        \end{cases}
\end{equation}
as $\P\rk{l_n=0}=0$ and $\P\rk{l_n=1}=\1\gk{n=0,1}$.

Recall that
\begin{equation}\label{EquationQOne}
    Q(x,z)=\sum_{j\ge 0}q_j(x)z^j\, .
\end{equation}
Now, we use Equation \eqref{EquationRelQ} to write
\begin{equation}\label{EquationQTwo}
    Q(x,z)=\prod_{j=1}^d Q(xp_j,z)+\sum_{k=0}^{d-2}\rk{z-z^2}\rk{1+\Fq(k)x}\ex^{-\Fq(k)x}\prod_{i=1}^k Q(x p_i,z)\, .
\end{equation}
Set
\begin{equation}\label{EquationQKXJ}
    Q(k,x,j)=\sum_{\mu\in \Pfrak_j\hk{k}}\prod_{i=1}^k q_{\mu_i}(xp_i)\, .
\end{equation}
Immediately
\begin{equation}
     Q(k,x,j)=\begin{cases}
         0 &\text{ if }j=0\, ,\\
         \sum_{i=0}^k \rk{1+p_i x}\ex^{-p_i x}&\text{ if }j=1\, .
     \end{cases}
\end{equation}
As $q_0(x)=0$, the largest value $\mu_i$ can take in Equation \eqref{EquationQKXJ} is $j-(k-1)$, as otherwise at least one of the other $\mu_j$'s has to be zero. This means that $Q(k,x,j)$ is a function of $\gk{q_i(z)}_{i=0}^{j-(k-1)}$.

Write $f_k(x)=\rk{1+\Fq(k)x}\ex^{-\Fq(k)x}$. Substituting Equation \eqref{EquationQOne} into Equation \eqref{EquationQTwo} yields (see Equation \eqref{IllustrationEq} for the mechanism)
\begin{multline}\label{EquationSystemOfEq}
    \sum_{j\ge 0}z^jq_j(x)\\
    =\sum_{j\ge 0}z^j\rk{Q(d,x,j)+\sum_{k=0}^{d-2} f_k(x)\rk{Q(k,x,j-1)\1\gk{j\ge 1}-Q(k,x,j-2)\1\gk{j\ge 2}}}\, .
\end{multline}
This system of equations is completely solvable for $q_j(x)$, as the coefficients on the right hand side of the depend only on $\gk{q_i(z)}_{i=0}^{j-1}$ for each $j$. Furthermore, the initial conditions for $q_j(x)$ are given in Equation \eqref{Initial conditions q_j}. This concludes the proof of Proposition \ref{PropositionMGF}.
\subsubsection{Collision Resolution Delay Analysis.}
In this section, we give a formula for the mean delay $\E\ek{t_2}$ caused by the resolution of the \ac{CRI} in steady state.

To calculate the expectation of $t_2$ given that the previous \ac{CRI} had length $n$, we do a case distinction: set $t_{2,m}$ the length of a tagged packet, given that there are $m$ other packages. Then
\begin{equation}
    \E\ek{t_2|c_k=n}=\sum_{m\ge 0}\E\ek{t_{2,m}}\ex^{-\l n}\frac{\rk{\l n}^m}{m!}=\sum_{m\ge 0}\sum_{k\ge 1}k\P\rk{t_{2,m}=k}\ex^{-x}\frac{x^m}{m!}\Big|_{x=\l n}\, ,
\end{equation}
which we abbreviate as $T_2(\l n)$.

Let $g\in\gk{1,\ldots,d}$ be the gate which the tagged packet joins. The evolution of $t_{2,m}$ is given by
\begin{equation}
    t_{2,m}=\begin{cases}
        1&\text{ if }m=0\, ,\\
        \1\gk{g<d}+\sum_{j=1}^{g-1} l_{I_j}+t_{2, I_g}&\text{ if }m\ge 1\, .
    \end{cases}
\end{equation}
Set
\begin{equation}
     G_{m+1}(z)=\E\ek{z^{t_{2,m}}}\qquad\text{and}\qquad G(x,z)=\sum_{m\ge 0}G_{m+1}(z)\ex^{-x}\frac{x^m}{m!}\, .
\end{equation}
We first state a proposition given a recursive equation for $G(x,z)$.
\begin{proposition}\label{ProposotionDelayG}
We have that
\begin{equation}\label{EquationGxz}
    G(x,z)=\sum_{k=1}^d p_k\rk{\ex^{-x}\rk{z-z^{k+\1\gk{k<d}}}+z^{\1\gk{k<d}}G(p_kx,z)\prod_{i=1}^{k-1}Q(p_ix,z)}\, ,
\end{equation}
where $Q$ is the moment generating function of $l_n$, as previously.
\end{proposition}
Before proving the above proposition, we explain how one can use it to obtain a formula for $T_2(x)$:

Taking the derivative with respect to $z$ at $z=1$ in Equation \eqref{EquationGxz}, we obtain
\begin{equation}
   T(x)=\sum_{k=1}^d p_k\rk{\ex^{-x}\rk{1-{k-\1\gk{k<d}}}+{\1\gk{k<d}}+T(p_kx)+\sum_{i=1}^{k-1}L(p_ix)}\, ,
\end{equation}
where $L(x)$ is the Poisson generating function for $L_n$, as in the proof of Corollary \ref{CorollaryLNbinomialFormula}.
Using $\alpha_n$ defined in Equation \eqref{Equationalphan}, the above implies that for $T(x)=\sum_{n\ge 0}t_nx^n$
\begin{equation}
    t_n=\frac{1}{n!}\frac{\sum_{k=1}^dp_k\rk{(-1)^{n+1}\rk{k-\1\gk{k=d}}+\alpha_n \sum_{i=1}^{k-1}p_i^n}}{1-\sum_{k=1}^d p_k^{n+1}}\, .
\end{equation}
As in \cite{yu2007high}, from the above equation, one can calculate $t_n$ and then numerically approximate the average delay $T_2(\l n)$.

\textbf{Proof of Proposition \ref{ProposotionDelayG}}:
Recall that $g\in\gk{1,\ldots,d}$ is the gate the tagged particle joins. Define
\begin{equation}
    G_{m+1}\hk{k}(z)=\E\ek{z^{t_{2,m}}|g=k}\, .
\end{equation}
Note that by doing a case distinction
\begin{equation}
    G_{m+1}(z)=\sum_{k=1}^d p_kG_{m+1}\hk{k}(z)\, .
\end{equation}
Furthermore, by conditioning that $\mu_i$ users join slot $i$, one obtains
\begin{equation}
     G_{m+1}\hk{k}(z)=z^{\1\gk{k<d}}\sum_{\mu\in\Pfrak_m\hk{d}}\binom{m}{\mu}p(\mu)G_{\mu_k+1}(z)\prod_{i=1}^{k-1}Q_{\mu_i}(z)\, .
\end{equation}
Recall that
\begin{equation}
    G(x,z)=\sum_{m\ge 0}G_{m+1}(z)\ex^{-x}\frac{x^m}{m!}\, .
\end{equation}
We can substitute to obtain 
\begin{multline}
        G(x,z)=\ex^{-x}z+\sum_{k=1}^d\sum_{m\ge 1}z^{\1\gk{k<d}}\\
        \cdot \sum_{\mu\in\Pfrak_m\hk{d}}\rk{\frac{G_{\mu_k+1}(z)p_k^{\mu_k+1}}{\mu_k!}\prod_{i=1}^{k-1}\frac{Q_{\mu_i}(z)\rk{p_i x}^{\mu_i}\ex^{-p_i x}}{\mu_i!}}\rk{\prod_{i=k+1}^d\frac{(p_ix)^{\mu_i}\ex^{-p_i x}}{\mu_i!}}\, ,
\end{multline}
Note that for $m=0$, the sum equals $p_k\ex^{-x}z^{k}$ and ehnce
\begin{multline}
    \sum_{m\ge 1}\sum_{\mu\in\Pfrak_m\hk{d}}\rk{\frac{G_{\mu_k+1}(z)p_k^{\mu_k+1}}{\mu_k!}\prod_{i=1}^{k-1}\frac{Q_{\mu_i}(z)\rk{p_i x}^{\mu_i}\ex^{-p_i x}}{\mu_i!}}\rk{\prod_{i=k+1}^d\frac{(p_ix)^{\mu_i}\ex^{-p_i x}}{\mu_i!}}\\
    =-p_k\ex^{-x}z^{k}+\sum_{m\ge 0}\sum_{\mu\in\Pfrak_m\hk{d}}\rk{\frac{G_{\mu_k+1}(z)p_k^{\mu_k+1}}{\mu_k!}\prod_{i=1}^{k-1}\frac{Q_{\mu_i}(z)\rk{p_i x}^{\mu_i}\ex^{-p_i x}}{\mu_i!}}\rk{\prod_{i=k+1}^d\frac{(p_ix)^{\mu_i}\ex^{-p_i x}}{\mu_i!}}\, .
\end{multline}
Now, we split the sum by first considering the subpartition $\gk{\mu_1,\ldots,\mu_k}$ whose cardinality we denote by $i$, and then the remaining partition $\gk{\mu_{k+1},\ldots,\mu_{d}}$ which consists of $d-k$ parts:
\begin{multline}
    \sum_{m\ge 0}\sum_{\mu\in\Pfrak_m\hk{d}}\rk{\frac{G_{\mu_k+1}(z)p_k^{\mu_k+1}}{\mu_k!}\prod_{i=1}^{k-1}\frac{Q_{\mu_i}(z)\rk{p_i x}^{\mu_i}\ex^{-p_i x}}{\mu_i!}}\rk{\prod_{i=k+1}^d\frac{(p_ix)^{\mu_i}\ex^{-p_i x}}{\mu_i!}}\\
    =\!\rk{\!\sum_{i=0}^\infty\!\sum_{\mu\in\Pfrak_i\hk{k}}\!\rk{\frac{G_{\mu_k+1}(z)p_k^{\mu_k+1}}{\mu_k!}\prod_{i=1}^{k-1}\frac{Q_{\mu_i}(z)\rk{p_i x}^{\mu_i}\ex^{-p_i x}}{\mu_i!}}\!}\!\sum_{m=0}^\infty\! \sum_{\mu\in \Pfrak_m\hk{d-k}}\!\rk{\prod_{i=k+1}^d\frac{(p_ix)^{\mu_i}\ex^{-p_i x}}{\mu_i!}} .
\end{multline}
Hence, we get that
\begin{equation}
    G(x,z)=\sum_{k=1}^d p_k\rk{\ex^{-x}\rk{z-z^{k+\1\gk{k<d}}}+z^{\1\gk{k<d}}G(p_kx,z)\prod_{i=1}^{k-1}Q(p_ix,z)}\, .
\end{equation}
This concludes the proof of Proposition \ref{ProposotionDelayG}.


\section{Discussions and Conclusion}
\label{sec:conclusion}
We have calculated the mean throughput, number of collisions, successes and idle slots for tree algorithms with successive interference cancellation. We have furthermore given a recursive relation which allows for approximations of to arbitrary order for the moment generating function of the \ac{CRI} length as well as the mean delay in steady state. We have shown numerically that a small reduction in throughput can lead to a bigger reduction in the number of collisions. However, our methods can be used for other observables of the random tree. We hence believe that by emulating our approach above, more properties of random tree algorithms can be calculated.


\bibliographystyle{abbrv}
\bibliography{Bibliography}

\end{document}